\documentclass[11pt,reqno]{amsart}
\usepackage{amsmath, amssymb, graphicx, mathrsfs, dsfont, hyperref}
\usepackage{epsfig}
\usepackage{palatino}
\usepackage{autonum}
\usepackage[makeroom]{cancel}
\usepackage{enumitem}  
\usepackage[usenames,dvipsnames,svgnames]{xcolor}
\usepackage[top=1in, bottom=1in, left=1in, right=1in]{geometry} 

\renewcommand{\Re}{\operatorname{Re}}
\renewcommand{\Im}{\operatorname{Im}}

\renewcommand{\max}{\operatorname{max}}

\newcommand{\s}{{\sigma}}

 \renewcommand{\a}{\alpha}
\renewcommand{\b}{\beta}
\newcommand{\e}{\epsilon}
\renewcommand{\d}{{\delta}}
\newcommand{\g}{\gamma}
\newcommand{\G}{\Gamma}

\renewcommand{\L}{\Lambda}

\renewcommand{\t}{\theta}

\newcommand{\z}{\zeta}

\newcommand{\twopartdef}[4]
{
	\left\{
		\begin{array}{ll}
			#1 & \mbox{if } #2 \\
			#3 & \mbox{if } #4
		\end{array}
	\right.
}

\theoremstyle{plain}
\newtheorem{theorem}{Theorem}[section]
\newtheorem{lemma}[theorem]{Lemma}

\usepackage{import}
\usepackage{xifthen}
\usepackage{pdfpages}
\usepackage{transparent}

\begin{document}

\title[Zeros of Dirichlet polynomials]{Zeros of Dirichlet polynomials}

\author[Arindam Roy]{Arindam Roy}

\address{Department of Mathematics and Statistics, University of North Carolina at Charlotte, 9201 University City Blvd., Charlotte, NC 28223, USA}

\email{arindam.roy@uncc.edu}
\author[Akshaa Vatwani]{Akshaa Vatwani}

\address{Department of Mathematics, 
Indian Institute of Technology Gandhinagar, 
Palaj, Gandhinagar
Gujarat 382355, India}

\email{akshaa.vatwani@iitgn.ac.in }



\subjclass[2010]{Primary 11M41; Secondary 11M26, 11N64.}
\thanks{\textit{Keywords and phrases.}  Dirichlet polynomials,  distribution of zeros, $k$-bounded  functions, $L$-functions, approximate functional equation}


\begin{abstract}
We consider a certain class of multiplicative functions $f: \mathbb N \rightarrow \mathbb C$ and study the distribution of zeros of  Dirichlet polynomials  $F_N(s)= \sum_{n\le N} f(n)n^{-s}$  corresponding to these functions.  We prove that the known non-trivial zero-free half plane for  Dirichlet polynomials associated to this class of multiplicative functions is optimal. We also  introduce a characterization of elements in this class based  on a new parameter depending on the Dirichlet series $F(s) = \sum_{n=1}^\infty f(n) n^{-s}$. In this context, we obtain non-trivial regions in which  the associated Dirichlet polynomials do have zeros. 

\end{abstract}

\maketitle

\thispagestyle{empty}

\section{Introduction}

Let $f:\mathbb{N}\to \mathbb{C}$ be an arithmetical function. 
The Dirichet series corresponding to $f$ is  given by 
\begin{align}\label{dir series}
F(s)=\sum_{n=1}^{\infty}\frac{f(n)}{n^s},
\end{align}
defined in the region of absolute convergence. Subject to certain additional conditions, such a Dirichlet series leads to an $L$-function via analytic continuation. The theory of $L$-functions in turn is intimately connected to arithmetical properties of integers in various settings.

One may also consider the finite Dirichlet series,  known as the  Dirichlet polynomial associated to $f$,  defined by 
\begin{align} \label{partial sum}
F_N(s)=\sum_{n=1}^{N}\frac{f(n)}{n^s}. 
\end{align} 
A fundamental connection between \eqref{dir series} and \eqref{partial sum} arises through the approximate functional equation. For instance, in a seminal paper  \cite{Chandrasekharan-Narasimhan},  Chandrasekharan and Narasimhan  show that if the Dirichlet series satisfies a general functional equation of the form 
\begin{align}
\Delta(s)F(s)=\Delta(\d-s)F(\d-s),
\end{align}
for some fixed  $\delta \in \mathbb R$,  $\Delta(s)=\prod_{\nu=1}^N\Gamma(\a_{\nu}s+\b_{\nu})$ with $\a_{\nu}>0$ and $\b_{\nu}$ complex, then under certain conditions,  one is led to an approximate functional equation  of the form
\begin{align}
F(s)=\sum_{n\leq x}\frac{f(n)}{n^s}+\frac{\Delta(\d-s)}{\Delta(s)}\sum_{n\leq y}\frac{f(n)}{n^{\d-s}}+R(x,y,|t|),
\end{align} 
where the term $R(x,y,|t|)$  depends upon the lengths $x,y$ of the Dirichlet polynomials and $t =\Im (s)$.  

Over the years,   number theory has  witnessed  an increasing interplay between Dirichlet polynomials and classical questions concerning $L$-functions. For instance, zero density results for $L$-functions were significantly improved by Halasz \cite{HalaszTuran-I, HalaszTuran-II} and Montgomery \cite{Montgomery-meanvalues, Montgomery-zeros} by relying on estimates on the frequency with which a certain Dirichlet polynomial could be ``large". In this context, Huxley \cite{Huxley-largevalues, Huxley-largevaluesII, Huxley-largevaluesIII}   introduced a ``reflection argument", which led to the question of how often two different Dirichlet polynomials could be simultaneously large. These methods were further developed by Jutila \cite{HuxleyJutila-largevaluesIV, Jutila-zerodens} to obtain new zero-density estimates for $L$-functions. 

Mean value theorems for Dirichlet polynomials are also of consequence towards proving mean value estimates for $L$-functions. Such techniques were first developed by K. Ramachandra  \cite{Ramachandra-fourthmean, Ramachandra-MVapplication, Ramachandra-MVapplication2} to treat fourth power means of certain $L$-functions. A number of non-trivial results have since been derived by refining and building upon these methods, for instance by Balasubramanian  \cite{BaluRamachandra-thmIngham, BaluRamachandra-thmMeurman},  Sankaranarayanan \cite{Sank-symmsquare},  Kamiya \cite{KamiyaYuichi},  among others.

Representing the Dirichlet series $F(s)$ in  terms  of the Dirichlet polynomial  has also proved instrumental towards estimating the order of  magnitude of $F(s)$ in the critical strip. 
A case of considerable importance  is when $f(n) \equiv 1$. In this case, in the half plane $\Re(s)>1$,  the Dirichlet series $F(s)$ represents the Riemann zeta function $\zeta(s)$. Understanding the growth of $\zeta(s)$ on the critical line $\Re(s)=1/2$ is a long-standing question. The Lindel\"of hypothesis predicts that 
\begin{align}
\z\left(\frac{1}{2}+it\right)\ll_{\e}|t|^{\e}, 
\end{align}
for any $\e>0$. 
Interestingly, one can reformulate this hypothesis in terms of suitable Dirichlet polynomials. In particular, Tur\'an \cite{TuranLindelof} showed that the bound 
\begin{align}
\left\lvert\sum_{n\leq N}\frac{(-1)^n}{n^{it}}\right\rvert\ll_\e N^{1/2+\e}|2+t|^{\e}
\end{align}
for any $\e>0$, is equivalent to the  Lindel\"of hypothesis. 

A famous conjecture in this context is the Riemann Hypothesis, which asserts that the nontrivial zeros of $\z(s)$ lie on the line $\Re(s)=1/2$. This hypothesis is stronger than  the Lindel\"of hypothesis, since it implies the bound 
\begin{align}
\z\left(\frac{1}{2}+it\right)\ll \exp\left(A\frac{\log t}{\log \log t}\right)
\end{align}
for some absolute constant $A$.  
Tur\'an's study of the surprising connections between properties of $\zeta(s)$ and those of  the associated Dirichlet polynomial  $\zeta_N(s)=\displaystyle\sum_{n\leq N}{n^{-s}}$ continued with his work in \cite{TuranApprox1,TuranApprox2}, where he investigated zeros of $\zeta(s)$.  
He proved that the assertion of the Riemann hypothesis implies that there exist positive constants $c_1,c_2$ such that there are no roots of $\zeta_N(s)$ in the region 
\[
\s\geq 1,  \quad c_1\leq t\leq \exp(e^{c_2\sqrt{\log N\log\log N}}), 
\] 
for all $N$ sufficiently large. In the reverse direction, he  showed that the  Riemann hypothesis is implied by the non vanishing of $\zeta_N(s)$ in the half strip  $\s\geq 1+N^{-1/2}(\log N)^3$, $\g_n\leq t\leq \g_N+e^{N^3}$ for some $\g_N\geq 0$ and all $N$ sufficiently large. In later work \cite{TuranApprox3}, he further explored these connections by weakening the hypothesis on zeros of $\zeta_N(s)$ and allowing it to fail for infinitely many $N$, so long as the number of indices $N\le x$ for which it fails is $o(\log x)$as $x\rightarrow \infty$. Going beyond such conditional  statements, in \cite{TuranApprox1}, he also showed that for $N$ sufficiently large, the Dirichlet polynomials $\zeta_N(s)$ do not vanish in the half plane 
\[
\s \ge 1+ \frac{ 2 \log \log N  }{\log N}. 
\] 
It was only in $2001$ that this zero-free region was improved by Montgomery and Vaughan \cite{MontgomeryVaughan} to 
\begin{equation}\label{MV zero free region}
\s \ge 1+  \left(\frac{4}{\pi}-1 \right) \frac{  \log \log N  }{\log N}. 
\end{equation}
Their method is general enough to be applied to completely multiplicative functions with values inside the unit disc, but fails for the more general class of multiplicative functions. 
In \cite{RoyVatwani}, the authors bridged this gap by  obtaining a zero-free region for partial sums  associated to   functions in a suitable subclass $\mathcal{C}_k$ of multiplicative functions. Let $k$ be a positive real number. We say that   $f\in \mathcal{C}_k$ if $f$ is multiplicative and  each of the series 
\begin{align}
F(s)=\sum_{n=1}^{\infty}\frac{f(n)}{n^s},\quad \log F(s)=\sum_{n=1}^{\infty}\frac{\L_f(n)}{(\log n)n^s}, \quad\text{and}\quad \frac{F'(s)}{F(s)}=\sum_{n=1}^{\infty}\frac{\L_f(n)}{n^s}\label{defck}
\end{align}
is absolutely convergent for $\Re(s)>1$ and $|\L_f(n)|\leq k\L(n)$.   
We will say that $f$ is $k$-bounded if $f \in \mathcal C_k$. This class of multiplicative functions includes most multiplicative functions studied in the literature and has been an object of active interest recently,  for instance in the work of Granville, Harper and Soundararajan \cite{GranvilleHarperSound2017}, 
as well as Drappeau, Granville and Shao \cite{GranvilleShao1, GranvilleShao2, DrappeauGranvilleShao}. 
  In \cite[Theorem 2.3]{RoyVatwani}, the authors proved that if $f\in \mathcal{C}_k$, then the corresponding Dirichlet polynomial $F_N(s)$ does not vanish in the half plane 
\begin{align}\label{new zfr}
\s\ge 1+\left(\frac{4k}{\pi}-1\right)\frac{\log \log N}{\log N},
\end{align}  
for $N$ sufficiently large. Specialization to the case $k=1$ allows us to  deduce the result \eqref{MV zero free region} of Montgomery and Vaughan.

Interestingly,  it was proved by Montgomery  \cite{MontgomeryApprox} that the  factor $\frac{4}{\pi} -1 $  appearing in \eqref{MV zero free region} is optimal. More precisely, he showed that for any positive constant $c < \frac{4}{\pi} -1$, there exists $N_0=N_0(c)$ such that if $N>N_0$, then $\zeta_N(s)$ \emph{does} have zeros in the half-plane
 \begin{align}
\s>1+c\frac{\log \log N}{\log N}. \label{zfr}
\end{align}
The question that we address in this paper is whether the constant $\frac{4k}{\pi}-1$ appearing in \eqref{new zfr} is also sharp.  Our main objective  is thus to investigate the zeros of $F_N(s)$ in the region 
\begin{align}
\s < 1+\left(\frac{4k}{\pi}-1\right)\frac{\log \log N}{\log N}.
\end{align}
While this seems a natural enough question to ask, the answer is more nuanced than one would expect.  We find that if we consider the class $\mathcal C_k$ as a whole, then the half-plane given by \eqref{new zfr} is indeed optimal. More precisely, given  $k>0$, we explicitly demonstrate a member of the class $\mathcal{C}_k$ such that for $0<c< 4k/\pi-1$,  the corresponding partial sums $F_N$ always have a zero in the half-plane $ \s>1+c\frac{\log \log N}{\log N}$ when $N$  is sufficiently large. 

While the existence of one such member of $\mathcal{C}_k$ is enough to establish optimality of the zero-free half plane \eqref{new zfr}; for a general element in $\mathcal{C}_k$ our result on zeros of its partial sums is somewhat surprising.  In the course of our proof,  we find that we need to take into account some additional pieces of information about  the Dirichlet series $F(s)$ associated to $f$.  This comes into play  in the form of a new parameter $m$, in addition to the existing parameter $k$. More precisely, we introduce a new characterization of elements of $\mathcal{C}_k$ by defining a subclass $\mathcal{C}_{k,m} \subseteq \mathcal{C}_{k}$, which we will describe shortly. Depending on the value of this parameter $m$, we are compelled to accept the possibility of elements in  $\mathcal{C}_k$, for whose partial sums  the   zero-free region \eqref{new zfr} may not be optimal.

In order to state our result precisely we first define the subclass $\mathcal{C}_{k,m}$.  Given $m >0$, we  say that  $f\in \mathcal{C}_{k,m}$ if $f\in \mathcal{C}_{k}$ and  we can write 
\begin{align}
F(s)=\frac{H(s)}{(s-1)^m},   \label{eq:pole}
\end{align}  
with  $H(s)$   analytic in 
\begin{align}\label{region}
\s>1-\frac{c_1}{(1+\log(|t|+2))^M}, 
\end{align}
for some  constants $c_1, M>0$. 
Moreover, we assume that  in the above region, $F$ satisfies the bounds  
\begin{align}
 \frac{F'(s)}{F(s)}\ll (\log (|t|+2))^B, \quad\text{and}\quad \log F(s)\ll (\log\log (|t|+2))^C, 
  \label{zerofrbdd}
\end{align}
for some $B, C>0$. It is worth noting that the inequality $m \le k$ is evident from \eqref{eq:pole} and the  definition of $\mathcal{C}_k$. 
With this setup, our main result can be stated as follows. 
\begin{theorem}\label{thm:zero}
Let $f\in \mathcal{C}_{k,m}$ and $F_N(s)$ denote the corresponding Dirichlet polynomial given by \eqref{partial sum}. Let $0<c<\frac{4m}{\pi}-1$. Then there is a constant  $N_0(c)$ such that  $F_N(s)$ has a zero in the region 
\begin{align}
\s>1+c\frac{\log\log N}{\log N}, 
\end{align} 
but does not vanish in the region 
\begin{align} \label{half plane}
\s\ge 1+\left(\frac{4k}{\pi}-1\right)\frac{\log \log N}{\log N},
\end{align} 
for all $N>N_0(c)$. 
\end{theorem}
Let us note  that the first part of this theorem is vacuous if $m \le \pi/4$. Hence, we may assume $m > \pi /4$ henceforth. 
We elaborate on two illustrative special cases of our result below. 
\begin{enumerate}
\item \textbf{The case $m=k$.} 
For Dirichlet polynomials  $F_N(s)$ corresponding to  the class $\mathcal{C}_{k,k}$ of multiplicative functions, Theorem \ref{thm:zero} gives us a zero-free region 
\[
\s \ge 1+  \left( \frac{4k}{\pi} -1  \right)\frac{  \log \log N }{\log N}, 
\]
for all $N$ sufficiently large, and also asserts that the value of the constant $4k/\pi -1$ above is sharp. 
Moreover, this class can be seen to be non trivial as follows. Let $k>0$ and $f(n)=d_k(n)$, where $d_k(n)$ denotes the multiplicative function defined by 
\[
d_k(p^m)=\binom{k+m-1}{m}
\] 
on prime powers  $p^m$.
Note that if $k$ is a positive integer, then $d_k(n)$ can be interpreted combinatorically as the number of ways of writing $n$ as  a product of $k$ positive integers. It is easy to see that $f\in \mathcal{C}_{k,k}$.

We have thus obtained an example of a function in $\mathcal{C}_k$ for which the zero-free half plane \eqref{half plane} corresponding to the Dirichlet polynomial is optimal. This establishes that the constant $\frac{4k}{\pi}-1$ obtained in Theorem 2.3 of  \cite{RoyVatwani} indeed cannot be improved. 

\item  \textbf{The case $m=1, k \in \mathbb{N}, k\ge 2$.}    We consider the Dedekind zeta function $\zeta_K(s)$ of a number field $K$ with degree $[K:\mathbb Q]$ equal to $k$.  For $\Re(s)>1$,  we have  
\[
\zeta_K(s) := \sum_{n=1}^\infty \frac{a(n)}{n^s}, 
\]
where $a(n)$ denotes the number of integral ideals $\mathfrak{a}$ of $K$ with norm $n$. Using standard theory of number fields, for instance  the discussion following equations $(2.2)$ and $(2.3)$ of M. Mine \cite{Mine}, it can be seen that $\zeta_K(s)$ is associated to the class $\mathcal{C}_k$. Moreover since $\zeta_K(s)$ can be  analytically continued to $\mathbb C$ except for a simple pole at $s=1$, we see that it is associated to the subclass $\mathcal{C}_{k,1}$. 
Consequently, Theorem \ref{thm:zero} yields a zero-free half plane \eqref{half plane} for the Dirichlet polynomial $\zeta_{K,N}(s)$, provided $N$ is sufficiently large.  However, the question of whether this region is optimal is left open since Theorem \ref{thm:zero} only asserts that the value of the  constant in \eqref{half plane} cannot be decreased below $\frac{4}{\pi}-1$, but gives no information about whether the value can be between $\frac{4}{\pi}-1$  and  $\frac{4k}{\pi}-1$. 
\end{enumerate}


This paper is organized as follows. In Section  \ref{lemmata}, we collect together the preliminary results that will be required later. In Section \ref{proof}, we give the proof of Theorem \ref{thm:zero}. 

\section{Lemmata} \label{lemmata}
We will use a quantative version of Perron's formula \cite[Theorem II.2.3]{tenenbaum}, which can be given as follows. 
\begin{lemma}\label{quanperr}
Let $f(n)$ be an arithmetical function, $\s_a$ be the abscissa of absolute convergence of the corresponding Dirichlet series $F(s)$. Then for any $\a>\max(0,\s_a)$, $T\geq 1$ and $x\geq 1$, we have 
\begin{align}
\sum_{n\leq x}f(n)=\frac{1}{2\pi i}\int_{\a-iT}^{\a+iT}F(s)\frac{x^s}{s}\,ds+\left(x^{\a}\sum_{n=1}^{\infty}\frac{|f(n)|}{n^{\a}(1+T|\log(x/n)|)}\right).
\end{align}
\end{lemma}

Let $r$ be a positive parameter. The Hankel contour $\mathcal{H}$ is a positively oriented contour consisting of  the circle $|s|=r$ excluding the arc from 
$re^{i(\pi - \epsilon)}$ to $r e^{i (-\pi +\epsilon) }$,  and of 
the horizontal lines from  $-\infty$ to $r e^{i(\pi - \epsilon)}$ and from 
$r e^{i(-\pi + \epsilon)}$ to $-\infty$.  
We may take $\e\to 0$. Let $\mathcal{H}(X)$ denote the part of the Hankel contour in the half plane $\s>-X$.  Hankel's formula   \cite[p.~180, Corollary II.0.18]{tenenbaum} for the contour $\mathcal{H}(X)$ is given as follows. 
\begin{lemma}\label{hankel}
Let $X>1$. Then we have uniformly for $z\in \mathbb{C}$
\begin{align}
\frac{1}{2\pi i}\int_{\mathcal{H}(X)}s^{-z}e^{s}\,ds=\frac{1}{\Gamma(z)}+O\left(47^{|z|}\Gamma(1+|z|)e^{-\frac{1}{2}X}\right).
\end{align}
\end{lemma}

We will also need the following lemma from \cite[p.~134, Lemma 2.3]{GranvilleHarperSound2017} which  follows from the work of Shiu \cite[Theorem 1]{Shiu1980}.
\begin{lemma}  \label{lem:shiu}
Let $k$ and $\e$ be fixed positive real numbers. Let $x$ be large, and suppose $x^{\e}\leq z\leq x$ and $q\leq z^{1-\e}$. Then uniformly for all $f\in \mathcal{C}_k$ we have 
\begin{align}
\sum_{\substack{x-z\leq n\leq x\\n\equiv a (\operatorname{mod} q)}}|f(n)|\ll_{k,\e}\frac{z}{\phi(q)}\frac{1}{\log x}\exp\left(\sum_{\substack{p\leq x\\p\nmid q}}\frac{|f(p)|}{p}\right),
\end{align}
where $a (\operatorname{mod} q)$ is any reduced residue class. 
\end{lemma}
Let $0\leq \d\leq \frac{1}{2}$ be fixed. We let $b_\delta(\theta)$ be a periodic function of period $1$, defined by 
\begin{align}
b_{\d}(\t)=\twopartdef{ie^{i\pi\t}}{\d\leq\t\leq 1-\d}{-e^{(1-(2\d)^{-1})i\pi\t}}{-\d<\t\leq\d.}
\end{align}
Then the Fourier series of $b_{\d}(\theta)$ is an absolutely convergent series  given by 
\begin{align}
b_{\d}(\t)=\sum_{j=-\infty}^{\infty}\hat{b}_{\d}(j)e^{2\pi i j\t},\label{fsbd}
\end{align}
where 
\begin{align} \label{fourier formula}
\hat{b}_{\d}(j)=\frac{\sin 2\pi\bigg(j\d-\frac{1}{2}\d+  { \frac{1}{4}  }
	\bigg)}  {(2j-1)2\pi\left(j\d-\frac{1}{2}\d+ { \frac{1}{4}  } \right)}.
\end{align}
Moreover as seen in \cite[Lemma 1]{MontgomeryApprox}, the Fourier coefficients satisfy the following properties.
\begin{lemma} \label{lem:fourier}
Let $0\leq \d\leq \frac{1}{2}$ be fixed. Then 
\begin{enumerate}[label=(\roman*)]
	\item the Fourier coefficient $\hat{b}_{\d}(j)\in \mathbb{R}$ for every integer $j$;
	\item $\hat{b}_{\d}(j)\ll_\delta (1+j^2)^{-1}$;
	\item $\displaystyle\lim_{\d\to 0+}\hat{b}_{\d}(j)=\frac{2}{\pi(2j-1)}=\hat{b}_{0}(j)$ for every integer $j$;
	\item $\hat{b}_{\d}(1)-\hat{b}_{\d}(0)$ is strictly decreasing in $\d$;
	\item $\displaystyle\hat{b}_{0}(j)\leq \frac{2}{\pi}$ for every integer $j$.
	\end{enumerate}
\end{lemma}
Let $a(n)$ be a completely multiplicative function defined by 
\begin{align}
a(p)=b_{\d}\left(\frac{1}{2\pi}\log p\right)\label{adefb}
\end{align}
for every prime number $p$. Clearly, $|a(n)|=1$. For any multiplicative function $f\in \mathcal{C}_{k,m}$, we define $g(n)=f(n)a(n)$. Consider the Dirichlet polynomial 
\begin{align}
G_N(s)=\sum_{n\leq N}\frac{g(n)}{n^s}.
\end{align}
Then using Bohr's equivalence theorem \cite[Theorems 8.12, 8.16]{Apostol}, we find that $F_N(s)$ and $G_N(s)$ take the same set of values in any half plane $\s>\s_0$. 
Consider the Dirichlet series 
\begin{align}
G(s)=\sum_{n=1}^{\infty}\frac{g(n)}{n^s}.
\end{align}
From \eqref{defck}, we see that the Dirichlet series $G(s)$ is absolutely convergent for $\Re(s)>1$. Moreover, using the notation of \eqref{defck} for $g$, we have $\Lambda_g(n)=\Lambda_f(n)a(n)$. Indeed, this follows by applying  the identity (cf. (52) of \cite{RoyVatwani})
\begin{equation}
g(n) \log n = \sum_{d|n} \Lambda_g(d)g(n/d)
\end{equation}
inductively to prime powers. The above observations allow us to conclude that $g \in \mathcal{C}_k$. 
Define 
\begin{align}
G^*(s):=\exp\left(\sum_{n=2}^{\infty}\frac{\Lambda_f(n)}{(\log n)n^{s}}b_{\d}\left(\frac{1}{2\pi}\log n\right)\right)
\end{align}
for $\Re(s)>1$. Then from \eqref{fsbd} we find
\begin{align}
G^*(s)&=\exp\left(\sum_{n=2}^{\infty}\frac{\Lambda_f(n)}{(\log n)n^{s}}\sum_{j=-\infty}^{\infty}\hat{b}_{\d}(j)n^{ij}\right)\\
&=\exp\left(\sum_{j=-\infty}^{\infty}\hat{b}_{\d}(j)\sum_{n=2}^{\infty}\frac{\Lambda_f(n)}{(\log n)n^{s-ij}}\right)\\
&=\prod_{j=-\infty}^{\infty}F(s-ij)^{\hat{b}_{\d}(j)},
\label{prodf}
\end{align} 
for $\Re(s)>1$. We expect the behaviour of $G^*(s)$ to be like that of $G(s)$. Indeed, if we define   
\begin{align}
\tilde{G}(s):=\frac{G(s)}{G^*(s)}=\exp\left(\sum_{n=2}^{\infty}\frac{\Lambda_f(n)}{(\log n)n^{s}}\left\{a(n)-b_{\d}\left(\frac{1}{2\pi}\log n\right)\right\}\right),
\label{tgs}
\end{align}
then one has the following result. 
\begin{lemma}\label{tilgb}
The function $\log \tilde{G}(s)$ is absolutely convergent in the half plane $\Re(s)\geq\frac{3}{4}$ and analytic in the half plane $\Re(s)>\frac{1}{2}$.
\end{lemma}
\begin{proof}
From \eqref{adefb} and the bound  $|\Lambda_f(n)|\leq k\Lambda(n)$ we have
\begin{align}
|\log\tilde{G}(s)|\leq 2k\sum_{m\geq 2}\sum_{p}\frac{1}{mp^{m\s}}\ll \sum_{p}\frac{1}{p^{2\s}}.
\end{align}
Since the final sum  above is analytic for $\Re(s)>\frac{1}{2}$ and uniformly bounded in $\Re(s)\geq \frac{3}{4}$, this completes the proof. 
\end{proof}

Let $K$ be a positive number which will be chosen later. We split the product \eqref{prodf} into two parts and write 
\begin{align}
G^*(s)=G_1(s)G_2(s),\label{g1g2}
\end{align}
where
\begin{align}
G_1(s)=\prod_{|j|\leq K}F(s-ij)^{\hat{b}_{\d}(j)}\quad\text{and}\quad G_2(s)=\prod_{|j|> K}F(s-ij)^{\hat{b}_{\d}(j)}. 
\label{defG1G2}
\end{align}
One has  the following asymptotic formula.
\begin{lemma}\label{g2estimate}
Let $f \in \mathcal{C}_{k,m}$, $F(s)$ be the associated Dirichlet series and $G_2(s)$ be given by $\eqref{defG1G2}$. Let $K> \log\log N$.  Then 
\begin{align}
G_2(s)=1+O_{\d,m}\left(\frac{\log\log N}{K}  \right), 
\end{align}
uniformly for  $\s\ge1+\frac{1}{\log N}$. 
\end{lemma}
\begin{proof}
For  {  $\s \ge 1+\frac{1}{\log N}$},  we write 
\begin{align}
\log G_2(s) &= \sum_{|j|>K} \hat b_\delta(j) \log F(s-ij)
\\
&\ll_\delta \sum_{|j|>K} (1+j^2)^{-1} m \log  \log N 
\\
&\ll_{\delta,m}  \frac{\log\log N}{K}, 
\end{align}
using \eqref{eq:pole} as well as  $(ii)$ of  Lemma \ref{lem:fourier}. Now, using   $e^x=1+O(x)$ for $|x|<1$,  we complete the proof. 
\end{proof}

\section{Proof of Theorem \ref{thm:zero}}\label{proof} 
Since the zero-free region \eqref{half plane} follows from Theorem 2.3 of \cite{RoyVatwani}, we  devote this section to proving the former part of Theorem \ref{thm:zero}.

Let $\s  := \Re(s) \ge 1+(2/\log N)$ and $|\Im s|\leq 1/2$. Using the quantitative version of Perron's formula given by Lemma \ref{quanperr}, we write
\begin{align}
\sum_{n\leq N}\frac{g(n)}{n^s}=\frac{1}{2\pi i}\int_{\a-iT}^{\a+iT}G(s+w)\frac{N^w}{w}\,dw+O\left(\frac{N^{\a}}{T}\sum_{n=1}^{\infty}\frac{|g(n)|}{n^{\a+\s}|\log((N+1/2)/n)|}\right)
\label{p1}
\end{align}
for $\a>0$ and large $T>0$. We rewrite the above integral as follows 
\begin{align}
\frac{1}{2\pi i}\int_{\a-iT}^{\a+iT}G(s+w)\frac{N^w}{w}\,dw=\sum_{n=1}^{\infty}\frac{g(n)}{n^s}\frac{1}{2\pi i}\int_{\a-iT}^{\a+iT}\left(\frac{N}{n}\right)^w\frac{dw}{w}.\label{si}
\end{align} 
Let $\t=1-\s+1/\log N$. Then by Cauchy's residue theorem we find 
\begin{align}
\frac{1}{2\pi  i}\int_{\a-iT}^{\a+iT}\left(\frac{N}{n}\right)^w\frac{dw}{w}=1+\frac{1}{2\pi i}\int_{\t-iT}^{\t+iT}\left(\frac{N}{n}\right)^w\frac{dw}{w}+E(T),\label{rt}
\end{align}
where $E(T)$ is the contribution from the horizontal integrals. It can be shown that 
\begin{align}
E(T)\ll \frac{N^{\a}}{Tn^{\a}\log(N/n)}.\label{ee}
\end{align}
We may now move the line of integration in \eqref{p1} from $\Re(w)=\a$ to $\Re(w)=\t$. 
Combining \eqref{p1}, \eqref{si}, \eqref{rt} and \eqref{ee}, we obtain 
\begin{align}
\sum_{n\leq N}\frac{g(n)}{n^s}=G(s)+\frac{1}{2\pi i}\int_{\t-iT}^{\t+iT}G(s+w)\frac{N^w}{w}\,dw+O\left(\frac{N^{\a}}{T}\sum_{n=1}^{\infty}\frac{|g(n)|}{n^{\a+\s}|\log((N+1/2)/n)|}\right).
\label{p2}
\end{align}

Letting $0<K<T$, let us  consider the integral 
\begin{align}
\frac{1}{2\pi i}\int_{\t {+} iK}^{\t+iT}\frac{x^w}{w}\,dw.
\end{align}
If $0<x<1$, then by the residue theorem, we may replace the line of integration $[\t+iK, \t+iT]$ by two horizontal integrals $[\t+iT, b+iT]$, $[\t+iK, b+iK]$ and a vertical integral $[b+iK,b+iT]$ for some $b>0$. On the vertical line $[b+iK,b+iT]$ the integrand is $\ll x^b/b$ and so the integral is of the order $ Tx^b/b$. Since $x<1$, this bound tends to $0$ as $b$ tends to $\infty$. On the top horizontal line the integrand is $\ll x^\s/T$ and so the integral is of the order $x^{\t}/(T\log x)$. Similarly, on the  lower horizontal line the integral is $\ll x^{\t}/(K\log x)$. Therefore, for $0<x<1$ we find that 
\begin{align}
\frac{1}{2\pi i}\int_{\t {+} iK}^{\t+iT}\frac{x^w}{w}\,dw\ll \frac{x^{\t}}{K\log x}.
\label{impint}
\end{align}
 If $x>1$, then we move the line of integration $[\t+iK, \t+iT]$ to the left and find that \eqref{impint}  holds in this case as well.

Let $K=\exp((\log\log N)^2)$ and  choose {$T$} large enough so that $T>K$. Then by \eqref{impint} we have 
\begin{align}
\frac{1}{2\pi i}\int_{\t+iK}^{\t+iT}G(s+w)\frac{N^w}{w}\,dw&=\sum_{n=1}^{\infty}\frac{g(n)}{n^s}\frac{1}{2\pi i}\int_{\t+iK}^{\t+iT}\left(\frac{N}{n}\right)^w\frac{dw}{w}\\
&\ll \frac{N^{1-\s}}{K}\sum_{n=1}^{\infty}\frac{|g(n)|}{n^{  \t+\s}|\log((N+1/2)/n)|}.\label{p3}
\end{align}
 We note here that $f\in \mathcal{C}_k$ implies that $|g(n)|= |f(n)| \le d_k(n)$  (cf. (10) of \cite{RoyVatwani}). 
 Making use of this bound, we will estimate the expression in \eqref{p3}  by considering two cases.
When $|n-N|>N/2$,  one has  $|\log((N+1/2)/n)|\gg 1 $. In this case, we get 
\begin{align}
\sum_{|n-N|>N/2}\frac{|g(n)|}{n^{ {\t} +\s}|\log((N+1/2)/n)|}\ll \sum_{n=1}^{\infty}\frac{d_k(n)}{n^{ { \t}+\s}}\ll (\log N)^k.\label{trvialbdd}
\end{align}
For the other case, we note that the interval $|n-N|\leq N/2$ is contained in the union of  intervals of the form $|n-N|\leq N/K$ and 
$2^{j}N/K\leq |n-N|\leq 2^{j+1}N/K$ for $j=0,1, \cdots, \lfloor\log_2 K\rfloor-1$. Then 
\begin{align}
\frac{N^{1-\s}}{K}\sum_{|n-N|\leq N/2}\frac{|g(n)|}{n^{\t+\s}|\log((N+1/2)/n)|}&\ll N^{-\s}\sum_{j=0}^{\lfloor \log_2 K\rfloor}\frac{2^{\t+\s}}{2^j}\sum_{|n-N|\leq 2^j N/K}|d_k(n)|\\
&\ll \frac{N^{1-\s}}{K}(\log N)^k\log K,\label{diadicbdd}
\end{align}
where in the ultimate step we have used Lemma \ref{lem:shiu}  with $q=1$, $d_k(p)=k+1$, and the known estimate 
\begin{align}
\sum_{p\leq x}\frac{1}{p}\ll \log\log x.
\end{align} 
Plugging \eqref{p3}, \eqref{trvialbdd} and \eqref{diadicbdd} into \eqref{p2}, recalling that $K= \exp((\log \log N)^2)$, and then letting $T$ tend to $\infty$, we find that 
\begin{align} 
\sum_{n\leq N}\frac{g(n)}{n^s}=G(s)+\frac{1}{2\pi i}\int_{\t-iK}^{\t+iK}G(s+w)\frac{N^w}{w}\,dw+O\left(N^{1-\s}\exp\left(-\frac{1}{2}(\log\log N)^2\right)\right).
\label{p4}
\end{align}
Now, from \eqref{tgs}, \eqref{g1g2} and Lemma \ref{g2estimate} we have
\begin{align}
G(s+w)&=\tilde{G}(s+w)G_1(s+w)\left(1+O\left(\frac{\log\log N}{K}\right)\right)\\
&=\tilde{G}(s+w)G_1(s+w)\left(1+O\left(\exp\bigg(-\frac{1}{2}(\log\log N)^2\bigg)\right)\right). 
\label{est for G}
\end{align}
Keeping in mind  that $\Re(s+w) = 1+1/\log N$,  using 
 \eqref{defck} and \eqref{defG1G2}  we obtain  
\begin{align}
\log G_1(s+w)&\ll \sum_{|j|\leq K}\hat{b}(j)\sum_{n=2}^{\infty}\frac{|\Lambda_f(n)|}{(\log n) n^{1+1/\log N}}\\
&\ll_{\d} k\log  \z\bigg(1+\frac{1}{\log N}\bigg)  \sum_{|j|\leq K}\frac{1}{j^2+1}\\
&\ll_{\d} k\log \log N, 
\label{g1bd}
\end{align}
where for the penultimate bound, we have used $|\Lambda_f(n)| \le k \Lambda(n)$ as well as ${(ii)}$ of Lemma \ref{lem:fourier}. 
Using \eqref{g1bd} and Lemma \ref{tilgb} in \eqref{est for G},  we have 
\begin{align}
G(s+w)
&=\tilde{G}(s+w)G_1(s+w)+O\left(\exp\bigg(-\frac{1}{4}(\log\log N)^2\bigg)\right).
\label{gg1}
\end{align}
The integral in \eqref{p4} is thus given by 
\begin{align} \label{intg in G1G1}
\frac{1}{2\pi i}\int\limits_{\t-iK}^{\t+iK}\!\!G(s+w)\frac{N^w}{w}\,dw=\frac{1}{2\pi i}\int\limits_{\t-iK}^{\t+iK}\!\!\tilde{G}(s+w)G_1(s+w)\frac{N^w}{w}\,dw
+O\left(N^{1-\s}\exp\left(-\frac{1}{6}(\log\log N)^2\right)\right)
\end{align} 

In order to deal with the integral on the right hand side above, we construct the following contour $\mathcal L$ (see Figure 1 below). 
\begin{figure}[ht]
\includegraphics[scale=1]{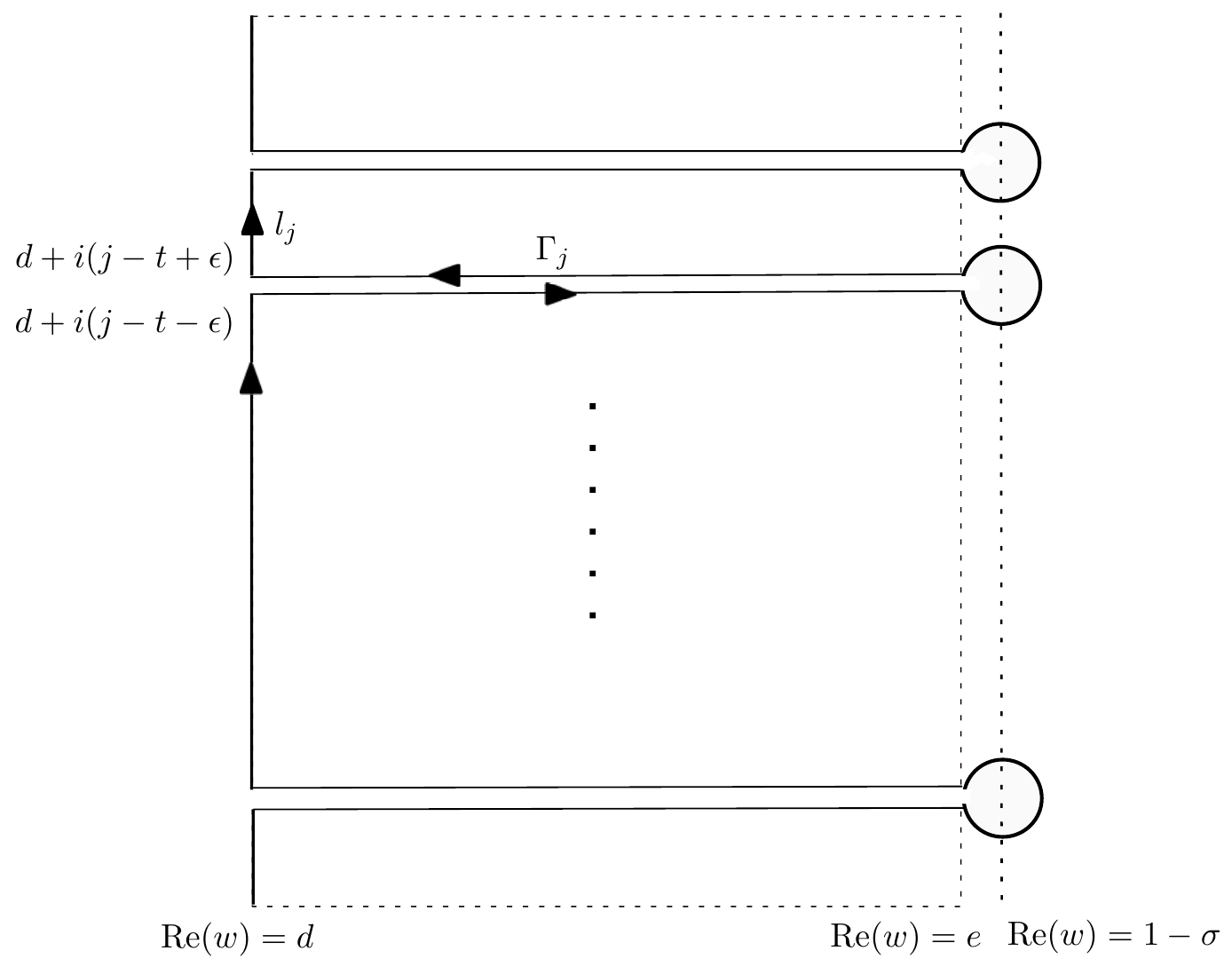}
\caption{The contour $\mathcal L$}
\end{figure}
Using the notation $s =\s+it$, let $d=1-\s-\frac{c_1}{3 (\log K)^M   }$ and $e=1-\s-\frac{1}{2\log N}$, where $c_1$ and $M$ are the constants from \eqref{region}.  Letting $j$ take integer values, we construct a  continuous curve $\mathcal{L}$ from $d-iK$ to $d+iK$ directed upward, by taking the union of line segments $l_j$  from $d+i(j-t+\e)$ to $d+i(j+1-t-\e)$, with any two consecutive line segments joined by a Hankel-type contour $\Gamma_j$ as follows.  The  contour $\Gamma_j$ consists of the horizontal line segment $[d+i(j-t-\e),e+i(j-t-\e)]$, followed by the counterclockwise circle $ |w-(1-s+ij)|=\frac{1}{2\log N}$ excluding the arc between  $e+i(j-t-\e)$ and $e+i(j-t+\e)$,  followed by  the horizontal line segment $ [e+i(j-t+\e),d+i(j-t+\e)] $. One may take $\e\to 0$. 

For $\s\geq 1+2/\log N$, the integrand $\tilde{G}(s+w)G_1(s+w)$ does not have any pole inside the positively oriented contour $\mathscr  C$  constructed from the line segments $[d-iK, \t-iK]$, $[\t-iK, \t+iK]$, $[\t+iK, d+iK]$ and the curve $-\mathcal{L}$. Therefore, by Cauchy's residue theorem, we have
\begin{align} \label{intg over L}
\frac{1}{2\pi i}\int_{\t-iK}^{\t+iK}\tilde{G}(s+w)G_1(s+w)\frac{N^w}{w}\,dw=
\frac{1}{2\pi i}\int_{\mathcal{L}}\tilde{G}(s+w)&G_1(s+w)\frac{N^w}{w}\,dw+E(K),
\end{align}
where $E(K)$ denotes the contribution from the horizontal integrals. By Lemma \ref{tilgb}, $\tilde{G}(s+w)$ is bounded on the horizontal sides of $\mathscr C$. From \eqref{zerofrbdd}, \eqref{defG1G2} and $(ii)$ of Lemma \ref{lem:fourier},  we have  
\begin{align}
\log G_1(s+w)\ll\sum_{|j|\leq K}\frac{(\log\log (K-j+4))^C}{1+j^2}\ll (\log\log K)^C, 
\end{align}
on the horizontal segments of $\mathscr C$. 
Thus for the contribution from horizontal integrals, we obtain
\begin{align} \label{Ek}
E(K)\ll \frac{\exp \big( (\log\log K)^C \big) }{K\log N}\ll \exp(-(\log\log N)^2).
\end{align}
We are now left to consider  the integral 
\begin{align} \label{intg}
\frac{1}{2\pi i}\int_{\mathcal{L}}\tilde{G}(s+w)&G_1(s+w)\frac{N^w}{w}\,dw 
\end{align}
in \eqref{intg over L}.
Fix $r \in \mathbb Z $, $|r|\le K$. 
We define 
\begin{align} \label{G_1,r} 
G_{1,r}(s+w)=\prod_{\substack{|j|\leq K\\j\neq r}} \big(F(s+w-ij) \big)^{\hat{b}_{\delta}(j)}.
\end{align}

For $\Re(s+w) \ge 1-\frac{c_1}{(\log K)^M}$, $|\Im(s+w) - r| \le 1/2$,  we have  
\begin{align}
\log G_{1,r}(s+w)&=\sum_{\substack{|j|\leq K\\j \neq r}}\hat{b}_{\delta}(j)\log F(s+w-ij)\\
&\ll \sum_{\substack{|j|\leq K\\j \neq r}} \frac{   \big(\log \log(r-j)^2+4 \big)^C}{j^2+1}\\
&\ll \big(\log\log(r^2+10) \big)^C, 
\label{g1k}
\end{align}
upon using \eqref{zerofrbdd}. 
From  \eqref{eq:pole}, we see that 
\begin{align}\label{laurent}
F(s+w-ir)\sim \frac{A}{(1-s-w+ir)^m}
\end{align}
for some constant $A$ as $\Re(s+w) \rightarrow 1^+$. 
On the vertical line segments of the contour $\mathcal L$, this gives 
\begin{align}
F(s+w-ir)\ll (\log K)^{c_0}, 
\end{align}
for some positive constant $c_0$. 
Using this along with the bound $\hat{b}_{\d}(r)\ll 1$,  Lemma \ref{tilgb} and \eqref{g1k}, the integral \eqref{intg} on the vertical line segments of $\mathcal{L}$ is 
\begin{align}
&\ll N^{1-\s-\frac{c_1}{3\log K}}\sum_{|r|\leq K}\int_{r-1}^{r+1}|F(s+d+it-ir)|^{b_\delta(r)}  \prod_{\substack{|j|\leq K\\j\neq r}}F(s+d+it-ij)|^{b_\delta(j) }   \frac{dt}{|d+it|}\\
&\ll N^{1-\s-\frac{c_1}{3\log K}}(\log K)^{O(1)} \sum_{|r|\leq K} 
\exp\left( (\log\log(r^2+10))^C\right)  \int_{r-1}^{r+1}\frac{dt}{|d+it|}\\
&\ll N^{\frac{-c_1}{3\log K}}(\log K)^{O(1)} \exp\left( (\log\log(K^2+10))^C\right) 
\sum_{|r|\leq K}  \int_{r-1}^{r+1}\frac{dt}{|d+it|} \\
&\ll \exp\{-(\log\log N)^6\}, \label{intg vert}
\end{align}
keeping in mind that $K$ was chosen to be $\exp( (\log \log N)^2)$, and $\s \ge 1+(2/\log N)$.  

We will now  estimate the integral \eqref{intg} over the remaining part of the contour $\mathcal L$, consisting of the Hankel type contours $\Gamma _r$.  We begin by estimating the integrand in a suitable way. 
Both $G_{1,r}(w)$ and $\tilde{G}(w)$ are analytic at $1-s+ir$.  Therefore as $|s+w-1-ir|\rightarrow 0$, using \eqref{laurent} one has
\begin{align} 
\!\!\!\frac{G_1(s+w)\tilde{G}(s+w)}{w}\!=\!  \frac{A_r G_{1,r}(1+ir)\tilde{G}(1+ir)}
{(1-s-w+ir)^{m\hat{b}_{\d}(r)}(1-s+ir)}\!+
O\!\left( \!\frac{|1-s+ir|^{-1} \exp\left( (\log\log(r^2+10))^C\right)} {|s+w-1-ir|^{m\hat{b}_{\d}(r)-1}}\!\right)
\label{lse}
\end{align}
for some constant $A_r$,  after using Lemma \ref{tilgb} and \eqref{g1k}. 

For each Hankel-type curve $\Gamma _r$, let us denote the circle $|w-(1-s+ir)|=1/ (2\log N)$ by $\gamma _r$, parametrized in the anticlockwise sense.   
In order to  estimate the contribution to the integral \eqref{intg} from the error term in \eqref{lse}, we consider   
\begin{align}
\int_{\G_ r}\frac{|N^w||dw|}{|s+w-1-i r|^{m\hat{b}_{\d}(r)-1}}&\ll \int_{d+i( r-t-\e)}^{e+i( r-t-\e)}
\frac{|N^w||dw|}{|s+w-1-ir|^{m\hat{b}_{\d}( r)-1}}
+
\int_{\gamma_r}  \frac{|N^w||dw|}{|s+w-1-ir|^{m\hat{b}_{\d}( r)-1}}
 \\
&\ll N^{-\s}\int_{1-\frac{c_1}{ 3(\log K)^M}-i\e}^{1-\frac{1}{2\log N}-i\e}\frac{|N^z||dz|}{|1-z|^{m\hat{b}_{\d}(r)-1}}
+ \int_{-\pi}^{\pi} \frac{N^{1-\s}  R |d\t|}{ R^{m \hat b_\d (r)-1 }   }
\end{align}  
where the final bound follows from the change of variable $z= w+s-ir$ for the first integral, and the parametrization $w =1-s+ir+ Re^{i\t}$ with $R= 1/(2\log N)$ for the second integral.  
Since the latter  is of the order $N^{1-\s}(\log N)^{m\hat{b}_{\d}(r)-2} $, upon applying the change of variable $\Re ({ 1-  } z)=  \frac{x} {  \log N }$, we obtain 
\begin{align}
\int_{\G_r}\frac{|N^w||dw|}{|s+w-1-ir|^{m\hat{b}_{\d}(r)-1}}&\ll N^{1-\s}(\log  N)^{m\hat{b}_{\d}(r)-2}\left(\int_{\frac{1}{2}}^{\frac{c_1\log N}{3(\log K)^M}}x^{1-m\hat{b}_{\d}(r)}e^{-x}\,dx+1\right)\\
&\ll_{m,r} N^{1-\s}   (\log  N)^{m\hat{b}_{\d}(r)-2}\, \Gamma\big(2-m\hat{b}_{\d}(r)\big).  
\label{error term intg}
\end{align}
To handle the contribution  to the integral \eqref{intg} from the main term in \eqref{lse}, we have 
\[\frac{1}{2\pi i}\int_{\G_r}\frac{N^w\,dw}{(s+w-1-ir)^{m\hat{b}_{\d}(r)}} = \frac{N^{1-s+ir}(\log N)^{m\hat{b}_{\d}(r)-1}}{2\pi i}\int_{\mathcal{H}\left(\frac{c_1\log N}  {3(\log K)^{ {M }}}
		\right)}z^{-m\hat{b}_{\d}(r)}e^z\,dz, 
\]
after applying the change of variable $z=(s+w-1-ir)\log N$.  Using Lemma \ref{hankel}, this integral can be written as 
\begin{align} \label{main term intg}
N^{1-s+ir}(\log N)^{m\hat{b}_{\d}(r)-1}\left(\frac{1}{\Gamma(m\hat{b}_{\d}(r))}
+O_{m,k}\left(\exp\bigg(-\frac{c_1\log N}{6(\log\log N)^ { {2M}} }\bigg)\right)\right).
\end{align}

We will now put together these estimates  as follows.  For 
 $0<c<\frac{4m}{\pi}-1$, let 
\begin{align}
\d=\frac{1}{50m}\left(\frac{4m}{\pi}-1-c\right)\label{delv} 
\end{align} 
 with $m>0$   as given in \eqref{eq:pole}.  Clearly, $\d < 4/(50\pi)$.
 For such $\delta$, using  \eqref{fourier formula}, 
  for $r\ne 1$ we have   $\hat{b}_{\d}(r)\leq \hat{b}_{\d}(1)-\frac{1}{6}$.

Putting together \eqref{lse}, \eqref{error term intg} and \eqref{main term intg}, we obtain that the contribution to \eqref{intg} from the Hankel-type contours is 
\begin{align}
\frac{1}{2\pi i}\sum_{|r|\leq K} \int_{\G_r}G_1(s+w)\tilde{G}(s+w)&\frac{N^w}{w}dw =\sum_{|r|\leq K}\frac{N^{1-s+ir}(\log N)^{m\hat{b}_{\d}(r)-1}A_rG_{1,r}(1+ir)\tilde{G}(1+ir)}{\G(m\hat{b}_{\d}(r))(1-s+ir)}\\
&+O\left( N^{1-\s}(\log N)^{m\hat{b}_{\d}(1)-2-m/6}\sum_{|r|\leq K}\frac{ \exp\left( (\log\log(r^2+10))^C\right)}{|1-s+ir|}\right)
\end{align}
The contribution from $r \neq 1$ can be absorbed into the error term above. Indeed, using Lemma \ref{tilgb} and \eqref{g1k} we have
\begin{align}
\frac{1}{2\pi i}\sum_{|r|\leq K}\int_{\G_r}G_1(s+w)\tilde{G}(s+w)&\frac{N^w}{w}dw=
\frac{N^{1-s+i}(\log N)^{m\hat{b}_{\d}(1)-1}A_1G_{1,1}(1+i)\tilde{G}(1+i)}{\G(m\hat{b}_{\d}(1))(1-s+i)}
\\
&+O\left( N^{1-\s}(\log N)^{m\hat{b}_{\d}(1)-1-m/6}\sum_{|r|\leq K}\frac{\exp\left( (\log\log(r^2+10))^C\right)}{|1-s+ir|}\right).
\end{align}
Recall that at the beginning of this section, we had assumed $\s \ge 1+ 2/ \log N$.  Let us consider now the region 
\begin{align}
 \s\geq 1+\frac{2}{\log N},  \quad\text{with} \quad |s-1|\leq \frac{1}{(\log N)^{2/3}}.
 \label{reg}
\end{align} 
We then have  
\begin{align}
\frac{1}{2\pi i}\sum_{|r|\leq K}\int_{\G_r}G_1(s+w)\tilde{G}(s+w)&\frac{N^w}{w}dw= 
\frac{N^{1-s+i}(\log N)^{m\hat{b}_{\d}(1)-1}A_1G_{1,1}(1+i)\tilde{G}(1+i)}{\G(m\hat{b}_{\d}(1))(1-s+i)}
\\
&+O\left( N^{1-\s}(\log N)^{m\hat{b}_{\d}(1)-1-m/6}  \exp\left( (\log\log(K^2+10))^{C+1}\right)
\right), 
\label{intg hankel}
\end{align}
giving the integral \eqref{intg} over the part of the contour $\mathcal{L}$ consisting of Hankel-type contours $\Gamma_r$. 
Putting together \eqref{p4}, \eqref{intg in G1G1}, \eqref{intg over L}, \eqref{Ek}, \eqref{intg vert} and \eqref{intg hankel}, we conclude that 
\begin{align}
\sum_{n\leq N}\frac{g(n)}{n^s}=G(s)+\frac{A_1N^{1-s+i}(\log N)^{m\hat{b}_{\d}(1)-1}G_{1,1}(1+i)\tilde{G}(1+i)}{\G(m\hat{b}_{\d}(1))(1-s+i)}\left(1+O((\log N)^{-m/7})\right).\label{finas}
\end{align}
Similar to our derivation of \eqref{lse},  in the region \ref{reg}, we have 
\begin{align}
G_1(s)\tilde{G}(s)=\frac{AG_{1,0}(1)\tilde{G}(1)}{(s-1)^{m\hat{b}_{\d}(0)}}+O\left(\frac{1}{|s-1|^{m\hat{b}_{\d}(0)-1}}\right), 
\end{align}
for some constant $A$. Then from Lemma \ref{tilgb}, \eqref{gg1}, \eqref{g1k} and \ref{reg}, we obtain 
\begin{align}
G(s)=\frac{AG_{1,0}(1)\tilde{G}(1)}{(s-1)^{m\hat{b}_{\d}(0)}}\left(1+O\left(\frac{1}{(\log N)^{2/3}}\right)\right).\label{gasym}
\end{align}

Let us now consider for $0<c<\frac{4m}{\pi}-1$, 
\begin{align}
\s_1=1+c\frac{\log\log N}{\log N}\quad\text{and}\quad \s_2=1+\left(\frac{8m}{\pi}-2-c\right)\frac{\log\log N}{\log N}, 
\end{align}
and let $\mathcal{R}$ be the rectangle with vertices $\s_i+it_i$, where $|t_1|\leq \frac{2\pi}{\log N}$ and $t_2=t_1+\frac{2\pi}{\log N}$.  We will proceed to show that $\sum_{n\le N} g(n)n^{-s}$ has a zero in $\mathcal{R}$.  
We write 
$$M(s)=M_1(s)+M_2(s),$$ where 
\begin{align}\label{M1M2}
M_1(s)=\frac{A_1N^{1-s+i}(\log N)^{m\hat{b}_{\d}(1)-1}G_{1,1}(1+i)\tilde{G}(1+i)}
{\G(m\hat{b}_{\d}(1))(1-s+i)}
\quad\text{and} 
\quad M_2(s)=\frac{AG_{1,0}(1)\tilde{G}(1)}{(s-1)^{m\hat{b}_{\d}(0)}}.
\end{align}
From \eqref{finas} and \eqref{gasym}, we obtain 
\begin{equation}\label{M(1+O)}
\sum_{n\le N} \frac{g(n)}{n^s} = M(s) \left(1+O\big(   (\log N)^{-m/7} \big)\right). 
\end{equation}
It is easy to see that on the right vertical side of $\mathcal{R}$, we have 
\begin{align}
\left\lvert\frac{M_1(s)}{M_2(s)}\right\rvert\ll_m \frac{|G_{1,1}(1+i)\tilde{G}(1+i)|(\log\log N)^{m\hat{b}_{\d}(0)}}{|(1-s+i)||G_{1,0}(1)\tilde{G}(1)|}(\log N)^{m(\hat{b}_{\d}(1)-\hat{b}_{\d}(0)-8/\pi)+c+1}.
\end{align} 
Moreover, using Lemma \ref{tilgb} and \eqref{g1k}, one has 
\begin{align}
1\ll \frac{|G_{1,1}(1+i)\tilde{G}(1+i)|}{|(1-s+i)||G_{1,0}(1)\tilde{G}(1)|}\ll 1\label{bsb}
\end{align}
in the region \eqref{reg}. From $(iii), (iv)$ of Lemma \ref{lem:fourier} as well as the choice of $\delta$ given by \eqref{delv}, we find that 
\begin{align}
\frac{8m}{\pi}-1-c> \frac{4m}{\pi}> m(\hat{b}_{\d}(1)-\hat{b}_{\d}(0)) { \geq \frac{4m}{\pi}-2m\pi\d^2} > \frac{4m}{\pi}-50m\d=1+c.\label{ineq1} 
\end{align}
Note that the penultimate inequality above follows by using \eqref{fourier formula} to obtain 
\begin{align}
\hat{b}_{\d}(1)-\hat{b}_{\d}(0) = \frac{2\cos (\pi \d)}{\pi }\left( \frac{2}{1-4\d^2}\right)
\end{align}
and observing that  $\cos (\pi \d) \ge 1- (\pi^2 \d^2 /2)$. 
\\
Thus, we see that on the right vertical side of the rectangle $\mathcal{R}$,
 $|M_1(s)/M_2(s)|\to 0$ as $N\to \infty$.  Writing 
 \[M(s)= M_2(s) \left(1+ O \left( \frac{|M_1(s)|}{|M_2(s)|}  \right)  \right),
 \] 
 we now consider the change in the argument of $M_2(s)$ on the right vertical side of $\mathcal{R}$. Keeping in mind  \eqref{reg} and the fact that  $\hat{b}_{\d}(0)<0$, we see from \eqref{M1M2} that change in the argument of $M_2(s)$ on the right side of $\mathcal{R}$ is very small as $N$ becomes sufficiently large. Therefore, change in the argument of $M(s)$ is also very small on the right side of $\mathcal{R}$. 
 
 Similarly, on the left vertical side of $\mathcal{R}$, we have 
\begin{align}
\left\lvert\frac{M_1(s)}{M_2(s)}\right\rvert\gg_m \frac{|G_{1,1}(1+i)\tilde{G}(1+i)|(\log\log N)^{m\hat{b}_{\d}(0)}}{|(1-s+i)||G_{1,0}(1)\tilde{G}(1)|}(\log N)^{m(\hat{b}_{\d}(1)-\hat{b}_{\d}(0))+1+c}, 
\end{align}
yielding $|M_2(s)/M_1(s)|\to 0$ as $N\to \infty$ on this side. For  $M_1(s)$, we again consider \eqref{M1M2} and note that when $N$ is sufficiently large,  change in the  argument of $1-s+i$ on $\mathcal{R}$ is very small.  On the other hand,  change in the  argument of $N^{1-s+i}$ is $2\pi$ on the left side of $\mathcal{R}$, as we go from $\s_1+it_1$ to $\s_1+it_2$. Thus we see that for $N$ sufficiently large,   change in the argument of $M(s)$ is close to $2\pi$ on the left side of $\mathcal{R}$. Arguing similarly, one finds that  change in the arguments for  $M_1(s)$ and $M_2(s)$, and hence for $M(s)$ is very small on the horizontal lines of $\mathcal{R}$. 
We may thus conclude that change in the argument of $M(s)$ along the boundary of the rectangle $\mathcal{R}$ is approximately $2\pi$. By the argument principle, this implies  that  $M(s)$ has at least one zero in the rectangle $\mathcal{R}$. Since \eqref{M(1+O)}  gives  
\begin{align}
\sum_{n\leq N}\frac{g(n)}{n^s}=M(s)+R(s), 
\end{align}
with  \begin{align}
 |R(s)|\ll |M(s)|(\log N)^{-m/7}<|M(s)|, 
 \end{align}
we may apply  Rouch\'e's theorem to conclude that $\sum_{n\leq N}\frac{g(n)}{n^s}$ has a zero in the rectangle $\mathcal{R}$. This completes the proof of the theorem.

\section*{Acknowledgments} 
The first author was partially supported by funds provided by the University of North Carolina at Charlotte. The second author was supported by the SERB-DST grant ECR/2018/001566 as well as the DST INSPIRE Faculty Award Program.     


\end{document}